\documentclass{amsart}
\usepackage{amssymb,latexsym,graphicx, amscd}
\newtheorem{theorem}{Theorem}[section]
\newtheorem{prop}[theorem]{Proposition}
\newtheorem{lemma}[theorem]{Lemma}

\newtheorem{question}[theorem]{Question}
\newtheorem{definition}[theorem]{Definition}
\newtheorem{cor}[theorem]{Corollary}

\begin{document}
\title[Tame/compatible four-dimensional Lie algebras]{A note on tame/compatible almost complex structures on four-dimensional Lie algebras}
\author{Andres Cubas}
\address{Department of Mathematics\\Florida International Univ.\\Miami, FL 33199}
\email{acuba001@fiu.edu}

\author{Tedi Draghici}
\address{Department of Mathematics\\Florida International Univ.\\Miami, FL 33199}
\email{draghici@fiu.edu}

 \begin{abstract}
 Four-dimensional, oriented Lie algebras $\mathfrak{g}$ which satisfy the tame-compatible question of Donaldson
 for all almost complex structures $J$ on $\mathfrak{g}$ are completely described. As a consequence,
 examples are given of (non-unimodular) four-dimensional Lie algebras with almost complex structures which are tamed but not compatible with symplectic forms.
 \end{abstract}
\date{March 16, 2015. This note is part of an undergraduate research project the first author conducted under the direction of the second author.}
\maketitle

\section{Introduction}
Among other interesting problems on compact 4-manifolds raised in \cite{D}, Donaldson asked the following:
\begin{question} \label{Donaldson} If $J$ is an almost complex structure tamed by a symplectic form, is $J$ also compatible
with a symplectic form?
\end{question}
\noindent Recall that an almost complex structure $J$ is said to be {\em tamed} by a
symplectic form $\omega$ (and such an $\omega$ is called $J$-{\em tamed}), if $\omega$ is positive on $J$-planes, i.e.
$$ \omega(u,Ju) > 0, \quad \mbox{for all vectors }  u \neq 0 . $$
An almost complex structure $J$ is said to be {\em compatible} with a
symplectic form $\omega$ (and such an $\omega$ is called {\em compatible} with $J$, or $J$-{\em compatible}), if $\omega$ is $J$-tamed and $J$-invariant, i.e.
$$ \omega(u,Ju) > 0 \mbox{ and } \omega(Jv,Jw) = \omega(v,w),  \quad \mbox{ for all vectors $u \neq 0$, $v$, $w$. } $$
Question \ref{Donaldson} is still open for compact 4-manifolds, although important progress has been made by Taubes \cite{Tau09}
who answered the question affirmatively for generic almost complex structures on 4-manifolds with $b^+ = 1$.
There are other significant positive partial results, e.g. see \cite{LZ2007}, \cite{LZ2010}, \cite{FLSV},
as well as results on the symplectic Calabi-Yau problem \cite{W, TWY, TW2}, also proposed by Donaldson in \cite{D} and known to imply
an affirmative answer to Question \ref{Donaldson} for compact 4-manifolds with $b^+ = 1$.
It is also worth noting that Donaldson's question is true locally for all almost complex 4-manifolds,
but this is no longer the case in higher dimensions, as for certain $J$'s the structure of their Nijenhuis
tensors becomes a local obstruction to the existence  of compatible symplectic forms  (see e.g. \cite{MT00}, \cite{T02}, \cite{Lej}).
There are no such obstructions for {\it integrable} almost complex structures and
an important version of Donaldson question (\cite{LZ2007}, p.678, and \cite{ST}, Question 1.7)
is whether it holds for compact complex manifolds of arbitrary dimensions. This is known to be true
for compact complex surfaces \cite{LZ2007}, for higher dimensions there are known only some partial results,
 e.g. \cite{EFV}, \cite{FKV}, \cite{EFG}, \cite{FK}.

\vspace{0.2cm}

Question \ref{Donaldson} has an obvious Lie algebra version, which has been already considered. Indeed, on a Lie algebra $\mathfrak{g}$ an almost complex structure is an endomorphism $J : \mathfrak{g} \rightarrow \mathfrak{g}$ with $J^2 = -1$, and we talk about symplectic (or closed, or exact) forms on $\mathfrak{g}$ with respect to the Chevalley-Eilenberg differential $d$ induced by the Lie bracket. Let us denote by $\mathcal{Z}^2$ the space of closed 2-forms and by $B_2$ the space of boundary 2-vectors on $\mathfrak{g}$. The space $B_2$ can be defined as the annihilator of $\mathcal{Z}^2$ with respect to the natural pairing between forms and vectors,
that is $u \in B_2$ if and only if $ \alpha(u) = 0$, for all $\alpha \in \mathcal{Z}^2$. For convenience, we also introduce the following definition.
We say that an oriented Lie algebra $\mathfrak{g}$ satisfies the {\it tame-compatible property}
if the answer to Question \ref{Donaldson} is affirmative for {\bf all} almost complex structures $J$ on $\mathfrak{g}$ inducing the given orientation.

\vspace{0.2cm}
For the 4-dimensional Lie algebra version of Question \ref{Donaldson}, one main result of Li and Tomassini in \cite{LT}
is the following:

\vspace{0.2cm}
\noindent {\bf Theorem:} (\cite{LT}, Theorem 0.2)   {\it On a four-dimensional Lie algebra $\mathfrak{g}$,
if the space of boundary 2-vectors $B_2$ is isotropic with respect to the wedge product,
that is, if $u \wedge u  = 0$, for all $u \in B_2$, then $\mathfrak{g}$ satisfies the tame-compatible property.}

\vspace{0.2cm}

As pointed out in \cite{LT}, any four-dimensional unimodular Lie algebra satisfies the condition of the theorem, thus,
satisfies the tame-compatible property. A consequence is that Question \ref{Donaldson} has an affirmative answer for
any left-invariant almost complex structure on a compact quotient of a 4-dimensional Lie group by a discrete subgroup
(Theorem 4.3, \cite{LT}). It is well known that if a Lie group admits a co-compact discrete subgroup
then its Lie algebra must be unimodular \cite{M}. Note also that the assumption in the above result is independent
of the choice of orientation on $\mathfrak{g}$, hence the conclusion is valid for both orientations.
Although it covers the important unimodular case, the above result of Li and Tomassini gives only a sufficient condition
for a 4-dimensional Lie algebra to satisfy the tame-compatible property. Our first observation was that the proof of Li-Tomassini
will go through under a slightly weaker condition which takes into account orientation.
Then we showed that our weaker condition is also
necessary for the tame-compatible property.
\begin{theorem} \label{main0}
Let $\mathfrak{g}$ be an oriented symplectic four-dimensional Lie algebra with a volume form $\mu$.
Then $\mathfrak{g}$ satisfies the tame-compatible property
if and only if the space of boundary 2-vectors $B_2$ is negative semi-definite
with respect to the bilinear form defined by the wedge product and the volume form,
that is, if and only if $\mu (u \wedge u) \leq 0$, for all $u \in B_2$.
\end{theorem}
As already mentioned above, for the proof of one direction, we could have slightly refined the arguments of Li-Tomassini
(with small adjustments, a version of the Theorem 2.5 \cite{LT} still holds).
However, partly to make our note self-contained and partly to present a slightly different proof, in section 3
we prefer to cast the 4-dimensional tame-compatible problem in an abstract linear algebra setting. We prove two linear algebra results
(Propositions \ref{lath1}, \ref{lath2}) which might have some independent interest. Theorem \ref{main0} follows directly from
Proposition \ref{lath1}, as shown in section 4.

\vspace{0.2cm}
Using the classification of four-dimensional symplectic Lie algebras obtained by Ovando \cite{Ov}, and her notations, there are two examples
(or, rather, one and one-half!) for which $B_2$ is not negative semi-definite.
\begin{cor} \label{cor1}
On the Lie algebra $\mathfrak{r}_2 \mathfrak{r}_2$ endowed with either orientation, or on the the Lie algebra $\mathfrak{ d }_{4,2}$ endowed with the non-complex
orientation there exist almost complex structures which are tamed by symplectic forms but which are not compatible with any symplectic forms.
These are the only 4-dimensional Lie algebras carrying such almost complex structures exist.
\end{cor}
In Section 4 we give the bracket descriptions of the two Lie algebras mentioned above.
Although Corollary \ref{cor1} follows directly from Theorem \ref{main0}, we provide in each case explicit examples of almost
complex structures which are tamed but not compatible. Let us just mention here that $\mathfrak{d}_{4,2}$ is the
Lie algebra underlying the unique proper 4-dimensional example of 3-symmetric space discovered by Kowalski \cite{Ko}.
With one orientation this Lie algebra admits a complex (in fact, K\"ahler) structure, with the other orientation it does not admit complex structures.
This is the orientation which we call ``non-complex''. Note also that $\mathfrak{d}_{4,2}$ does admit symplectic structures for both orientations.

\vspace{0.2cm}

To end the introduction, let us note that it was well known that the Lie algebra version of Question \ref{Donaldson}
can have negative answer for almost complex structures on Lie algebras of dimension 6 or higher, even in the nilpotent case.
The first such examples are due to Migliorini and Tomassini \cite{MT00} (see also \cite{T02}, \cite{AT}, \cite{A-thesis}).
It would be interesting to know if Theorem \ref{main0} extends in any way for Lie algebras of dimensions higher than 4.
Certainly, on an abelian Lie algebra of arbitrary even-dimension
all almost complex structures are tamed and compatible with symplectic forms. As a direct consequence of Lemma 3.2 in \cite{EFV}, we observe
that this property is specific {\bf only} to abelian Lie algebras (see Proposition \ref{main1}). We leave open
the question of the existence of non-abelian Lie algebras of dimension greater or equal to 6 which satisfy the tame-compatible property
(for all almost complex structures) and an eventual classification of such examples.

 \vspace{0.2cm}

{\bf Acknowledgments:} The second author is grateful to Tian-Jun Li
for encouragement to write this note and for useful comments on earlier versions.
He also thanks Anna Fino for helpful observations.

\section{Notations and Preliminaries}

Given a Lie algebra $\mathfrak{g}$ of dimension $4$, we denote by $\Lambda^k(\mathfrak{g})$ and  $\Lambda^k(\mathfrak{g}^*)$, respectively,
the spaces of (real) $k$-vectors and $k$-forms on $\mathfrak{g}$. The Lie bracket induces the Chevalley-Eilenberg
differential $d$ on the spaces of forms on $\mathfrak{g}$. On $\Lambda^1(\mathfrak{g}^*)$, $d$ is defined
by
$$ d\alpha (u, v) = - \alpha([u, v]), \; \; \alpha \in \Lambda^1(\mathfrak{g}^*), \; u,v \in \mathfrak{g}, $$
and then is extended to $d: \Lambda^k(\mathfrak{g}^*) \rightarrow \Lambda^{k+1}(\mathfrak{g}^*)$ by the Leibniz rule. Using the non-degenerate pairings
$$ \Psi : \Lambda^k(\mathfrak{g}) \times \Lambda^k(\mathfrak{g}^*) \rightarrow \mathbb{R},  \; \; \Psi (u, \alpha) = \alpha(u) \; , $$
one defines the differential $d$ for $k$-vectors as well, $d: \Lambda^k(\mathfrak{g}) \rightarrow \Lambda^{k-1}(\mathfrak{g})$. If $u$ is a $k$-vector,
$du$ is defined (uniquely, since $\Psi$ is non-degenerate) by
$$ \Psi( du, \alpha) := (-1)^{k-1} \Psi(u, d\alpha) \; , \; \forall \alpha \in \Lambda^{k-1}(\mathfrak{g}^*) \; . $$
Thus one obtains the dual Chevalley-Eilenberg complexes
$(\Lambda^*(\mathfrak{g}^*), d)$, $(\Lambda^*(\mathfrak{g}), d)$, yielding the real cohomology, respectively homology of
the Lie algebra. In this note we will only need the spaces of closed 2-forms, and boundary 2-vectors
$$\mathcal{Z}^2 = \{ \alpha \in  \Lambda^2(\mathfrak{g}^*) \; | \; d \alpha = 0 \} , \; \; B_2 = \{u \in \Lambda^2(\mathfrak{g}) \; | \; \exists w \in \Lambda^3(\mathfrak{g}), \; u = dw \} .$$
Equivalently, $B_2$ can be seen as the annihilator of $\mathcal{Z}^2$ with respect to the pairing $\Psi$ above,
$$B_2 = \{ u \in \Lambda^2(\mathfrak{g}) \; | \;  \Psi(u,\alpha) = 0, \forall \alpha \in \mathcal{Z}^2 \} \; .$$

Suppose from now on that the four-dimensional Lie algebra $\mathfrak{g}$ is oriented by
a fixed volume form $\mu \in \Lambda^{4}(\mathfrak{g}^{\star})$.
An important feature of dimension 4 is that
the wedge product and the fixed volume form induce non-degenerate inner products of signature (3,3) on the 6-dimensional spaces of 2-vectors or 2-forms:
\begin{equation} \label{phiupmu}
 \Phi_{\mu} : \Lambda^2(\mathfrak{g}) \times \Lambda^2(\mathfrak{g}) \rightarrow \mathbb{R} , \; \; \;  \;
\Phi_\mu (u, v) = \mu(u \wedge v) , \; \; \forall \; u, v \in \Lambda^2(\mathfrak{g}) .
\end{equation}
\begin{equation} \label{phimu}
 \Phi^{\mu} : \Lambda^2(\mathfrak{g}^{\star}) \times \Lambda^2(\mathfrak{g}^{\star}) \rightarrow \mathbb{R}, \; \; \;  \;
\alpha \wedge \beta = \Phi^\mu (\alpha, \beta) \mu , \; \; \forall \; \alpha, \beta \in \Lambda^2(\mathfrak{g}^{\star}) .
\end{equation}
It can be easily checked that these are dual inner products,
that is the Riesz maps between $\Lambda^2(\mathfrak{g}^{\star}) = (\Lambda^2(\mathfrak{g}))^{\star}$ and $ \Lambda^2(\mathfrak{g})$ induced by each inner product
are isometries, inverse to each other (see also Lemma 1.5 in \cite{LT}). With these Riesz maps, the space of boundary 2-vectors $B_2$, is identified with the orthogonal complement
$(\mathcal{Z}^2)^{\perp}$ of $\mathcal{Z}^2$ in $\Lambda^2(\mathfrak{g}^{\star})$ with respect to $\Phi^{\mu}$.

Let now $J$ be an almost complex structure on the Lie algebra $\mathfrak{g}$, that is,
an endomorphism $J : \mathfrak{g} \rightarrow \mathfrak{g}$ with $J^2 = -1$. The induced action of $J$ on the bundle of 2-forms,
$\alpha(\cdot, \cdot) \rightarrow \alpha(J\cdot, J\cdot)$, is an involution. This induces the decomposition of $\Lambda^2(\mathfrak{g}^{\star})$
into $J$-invariant and $J$-anti-invariant forms, respectively, the $\pm 1$-eigenspaces of the above action
$$\Lambda^2(\mathfrak{g}^{\star}) = \Lambda_J^+(\mathfrak{g}^{\star}) \oplus \Lambda_J^-(\mathfrak{g}^{\star}) \; .$$
This decomposition is orthogonal with respect to the bilinear form $\Phi^{\mu}$. The particularity of dimension 4 is that in this case
$\Lambda_J^-(\mathfrak{g}^{\star})$ is a two-dimensional plane positive-definite with respect to $\Phi^{\mu}$. Certainly, $\Lambda_J^+(\mathfrak{g}^{\star})$
with the restriction of $\Phi^{\mu}$ becomes a Minkowski vector  space of signature $(1,3)$ (one ``+'', three ``-'').
Similarly, for 2-vectors, there is the decomposition
$$\Lambda^2(\mathfrak{g}) = \Lambda_J^+(\mathfrak{g}) \oplus \Lambda_J^-(\mathfrak{g}) \; ,$$
and with the natural identifications through the Riesz maps above, we have
$$ (\Lambda_J^-(\mathfrak{g}^{\star}))^{\perp} = \Lambda_J^+(\mathfrak{g}), \; \; \; (\Lambda_J^+(\mathfrak{g}^{\star}))^{\perp} = \Lambda_J^-(\mathfrak{g}) \; .$$
As observed by Donaldson in the introduction of \cite{D}, in dimension four various geometric structures can be characterized
in terms of subspaces of the space of 2-forms (or 2-vectors) and their behavior with respect with the above bilinear forms.
The following proposition gathers, in the Lie algebra context, some of the observations made in the introduction of \cite{D}.
\begin{prop} \label{phimu4d} Let $\mathfrak{g}$ be a four-dimensional Lie algebra, oriented by
a fixed volume form $\mu \in \Lambda^4(\mathfrak{g}^{\star})$.
\begin{itemize}
\item A Riemannian metric on $\mathfrak{g}$ is given (up to rescaling) by a 3-dimensional subspace $\Lambda^+ \subset \Lambda^2(\mathfrak{g}^{\star})$
on which $\Phi_{\mu}$ is positive definite.
\item A (positively oriented) symplectic form is defined by an element $\omega \in \Lambda^2(\mathfrak{g}^{\star})$, with $d \omega = 0$ and
$\Phi_{\mu}(\omega, \omega) > 0$.
\item The map $J \rightarrow \Lambda^-_J$ is a two-to-one correspondence between (positively oriented) almost complex structures $J$ on $\mathfrak{g}$
and 2-dimensional planes in $\Lambda^2(\mathfrak{g}^{\star})$, positive definite with respect to $\Phi^{\mu}$;
\item An almost complex structure $J$ is tamed by a symplectic form $\omega$ if $\Lambda^-_J + \mathbb{R} \omega$ generates
a 3-dimensional positive definite subspace of  $\Lambda^2(\mathfrak{g}^{\star})$
on which $\Phi_{\mu}$ is  positive definite.
\item An almost complex structure $J$ is compatible with the symplectic form $\omega$ if $J$ is tamed by $\omega$ and
 $\omega$ is orthogonal to $\Lambda^-_J$ with respect to  $\Phi_{\mu}$.
 \end{itemize}
\end{prop}
For the third point, the correspondence is two-to-one, as $J$ and $-J$ induce the same plane of anti-invariant forms. Choosing a positive definite 2-plane
$H \subset (\Lambda^2(\mathfrak{g}^{\star}), \Phi^{\mu})$, one can also determine the sign of the almost complex structure by additionally choosing
one component of null cone in the Minkowski space $H^{\perp}$ (this amounts to choosing the simple positive $J$-vectors between $v\wedge Jv$ vs. $v \wedge (-J)v$).
Certainly, the distinction $J$ vs. $-J$ is irrelevant with respect to the tame and compatible properties, as $J$ being tamed (compatible)
is equivalent to $-J$ being tamed (compatible).

\vspace{0.2cm}
Because of the above proposition, it is natural to consider the extension of the Lie algebra 4-dimensional tame/compatible problem to an abstract linear algebra setting, as we do in the next section.

\section{The linear algebra extension}
Let $(V, \langle \, \cdot\; , \; \cdot \, \rangle)$ be a pseudo-Euclidean real vector space of signature $(k, l)$, with $k \geq 2$, $l \geq 2$.
In other words, $\langle \, \cdot\; , \; \cdot \, \rangle$ is a real, symmetric, non-degenerate bilinear form on $V$, which, when diagonalized
by Sylvester's theorem, yields a diagonal form $(+1, \; ... \; , \; +1, \; -1, \; ... \; , \; -1)$ with $k$ plus ones and $l$ minus ones.

For any subspace $L \subseteq V$, denote by $q^+(L), q^-(L), q^0(L)$ the Sylvester's numbers of $(L, \langle \, \cdot\; , \; \cdot \, \rangle|_{L \times L})$, i.e.
respectively the number of plus ones, the number of minus ones and the number of zeros occurring in a diagonalization of $\langle \, \cdot\; , \; \cdot \, \rangle|_{L \times L}$. We'll also denote
by ${\rm dim}(L)$ the dimension of the subspace $L$, and we denote by $L^{\perp}$ the orthogonal subspace of $L$. For brevity, we'll call a subspace $L$ be
a positive $r$-plane, if $L$ is $r$-dimensional and positive definite with respect to the inner product.

\vspace{0.2cm}

Consider now a subspace $Z \subseteq V$ with $q^+(Z) \geq 1$ and consider also a positive $(k-1)$-plane $H \subseteq V$.
In other words, $ {\rm dim }(H) = q^+(H) = k-1$, where, by assumption, $k = q^+(V)$.
We introduce the following definitions:
\begin{definition} (i) We say that $H$ is $Z$-extendable if there exists $z \in Z$ so that $H + \mathbb{R}z$ is a positive
$k$-dimensional plane in $V$.

(ii) We say that $H$ is $Z$-orthogonally-extendable if there exists $z \in Z \cap H^{\perp}$ so that $H + \mathbb{R}z$ is a positive
$k$-dimensional plane.
\end{definition}

The goal of this section is to prove the following two results:
\begin{prop} \label{lath1} Suppose $(V, \langle \, \cdot\; , \; \cdot \, \rangle)$ is a pseudo-Euclidean real vector space of signature $(k, l)$, with $k \geq 2$, $l \geq 2$,
and suppose $Z \subseteq V$ is a subspace with $q^+(Z) \geq 1$. The following statements are equivalent:

(i) $q^+(Z^{\perp}) = 0$;

(ii) for any $H$ positive $(k-1)$-plane, if $H$ is $Z$-extendable, then $H$ is also  $Z$-orthogonally-extendable.
\end{prop}

\begin{prop} \label{lath2} Suppose $(V, \langle \, \cdot\; , \; \cdot \, \rangle)$ is a pseudo-Euclidean real vector space of signature $(k, l)$, with $k \geq 2$, $l \geq 2$,
and suppose $Z \subseteq V$ is a subspace of $V$. The following statements are equivalent:

 (i) $q^+(Z) = q^+(V) = k$;

 (ii) any positive $(k-1)$-plane $H$ is $Z$-extendable;

 (iii) any positive $(k-1)$-plane $H$ is $Z$-orthogonally-extendable.
\end{prop}

\subsection{A lemma for a Minkowski vector space}
Let $(W, \langle \, \cdot\; , \; \cdot \, \rangle)$ be a Minkowski vector space with dimension at least 3, i.e. $\langle \, \cdot\; , \; \cdot \, \rangle$ is a real, symmetric,
non-degenerate inner product of signature $(1, l)$ with $l \geq 2$. The convention we adopt for a Minkowski vector space is that when diagonalized
the inner product has one plus one and the rest are minus ones.

Let $C(W) = \{ w \in W \; | \; \langle  w, w\rangle  \; \geq 0 \}$ be the set of {\it causal vectors} in $W$, i.e. {\it time-like vectors} ($w$, so that $\langle  w, w\rangle  \; > 0$) or {\it null vectors}
($w$, so that $\langle  w, w\rangle  = 0$). Note that $C(W)$ is a closed convex cone. Without the origin, $C(W)\setminus \{ 0 \}$ has two connected components, which we denote
$C^+$ and $C^-$. One notes immediately that $C^- = - C^+$. Of course, which of the two components we denote $C^+$ and which we denote $C^-$ is just a convention.
To give some motivation of this notation, note that if we fix an orthogonal basis for $W$,  $\{ e_0, \; e_1, \; ... \; , \; e_l \}$, with $\langle e_0,e_0\rangle  = 1$,
$\langle e_j,e_j\rangle  = - 1$, for $j = \{1, ... , l\}$ and $\langle e_a, e_b\rangle  = 0$ for $a\neq b \in \{ 0, 1, ... , l\}$, then we can choose $C^+(W)$, by convention, to be the set of causal vectors
$w \in C(W)$ with $\langle w, e_0\rangle  \; > 0$. Certainly the choice depends on the basis. If we replace $e_0$ by $-e_0$ in the fixed basis, we'll obviously pick the other component as $C^+$.

\begin{lemma} \label{mink-lemma} (i) For any non-zero causal vectors $u, v$ in the same connected component (e.g. $u, v \in C^+(W)$), we have $\langle  u, v\rangle  \; \geq 0$ with equality if and only if
$u$ and $v$ are proportional null vectors. If $u, v$ are non-zero causal vectors in different components, then $\langle  u, v\rangle  \; \leq 0$ with equality if and only if
$u$ and $v$ are proportional null vectors.

(ii) If $u \in W$, $u \neq 0$, satisfies $\langle  u, v \rangle  \; \geq 0$ for any $v \in C^+(W)$, then $ u \in C^+(W)$.

(iii) If $L \subseteq W$ is a subspace which contains no non-zero causal vector, then $L^{\perp}$ must contain a time-like vector.

\end{lemma}

\begin{proof} Part (i) follows from Cauchy-Schwarz inequality and is a standard fact for Minkowski vector spaces (e.g. see \cite{PR1}, p.3-4). Part (ii)
can be also easily checked and we leave it to the reader. It might be also well known in the literature. For part (iii), note that
if $L$ does not contain any non-zero causal vector, it means that $L$ is negatively defined with respect to the inner product. In particular,
$L$ is non-degenerate so $L \oplus L^{\perp} = W$. Thus, $L^{\perp}$ must contain a time-like vector.
\end{proof}

\subsection{Equivalent characterizations of the $Z$-extendable notions} As in the start of this section,
let $(V, \langle \, \cdot\; , \; \cdot \, \rangle)$ be a pseudo-Euclidean real vector space of signature $(k, l)$, with $k \geq 2$, $l \geq 2$.
Let $Z$ be a subspace of $V$ with $q^+(Z) \geq 1$ and let $H$ be a positive $(k-1)$-dimensional plane in $V$.
In this subsection we'll prove some equivalent characterizations of the $Z$-extendable and $Z$-orthogonally extendable notions.

Note first that since $H$ is a positive $(k-1)$-dimensional plane, $H$ is in particular non-degenerate, so
$$ V = H \oplus H^{\perp}, $$
 where $H^{\perp}$ with the induced inner product is a Minkowski space of signature $(1,l)$. We will apply Lemma 1 to $H^{\perp}$,
 so $C(H^{\perp})$ will denote the set of causal vectors in $H^{\perp}$ and $C^+(H^{\perp})$ will be one connected component
 of $C(H^{\perp}) \setminus \{ 0 \}$. Denote also
 $$ \pi^{H} : V \rightarrow H, \; \; \pi^{H^{\perp}} : V \rightarrow H^{\perp} \; , $$
 the corresponding projections.

 \begin{lemma} \label{Zext-lemma} With the notations above, the following are equivalent:

 (i) $H$ is $Z$-extendable;

 (ii) $q^+(H + Z) = k$;

 (iii) $q^+((H+Z) \cap H^{\perp}) \geq 1$;

 (iv) There exists $z \in Z$ such that $\langle z, u\rangle  \; \; > 0$, for all $u \in C^+(H^{\perp})$;

 (v) $Z^{\perp} \cap C^+(H^{\perp}) = \emptyset$.

\end{lemma}

\begin{proof} The equivalences $(i) \Leftrightarrow (ii) \Leftrightarrow (iii)$ are clear. We show next that $(iii) \Rightarrow (iv) \Rightarrow (v) \Rightarrow (iii)$.
For the first implication, assuming $(iii)$, there exist $h \in H$, $z \in Z$, such that $h +z$ is a time-like vector in $H^{\perp}$, that is
$h + z \in H^{\perp}$ and $\langle h+z\; , \; h+z\rangle  \; > 0$. Without loss of generality, we can assume that $h+z$ is in the "positive" component $C^+(H^{\perp})$.
By Lemma 1 (i), we have
$$ \langle  h+ z, u \rangle  \; > 0, \; \; \forall u \in C^+(H^{\perp}), \mbox{ hence } \langle  z, u \rangle  \; > 0, \; \; \forall u \in C^+(H^{\perp}) , $$
as $h \in H$.

The implication $(iv) \Rightarrow (v)$ is immediate, since $\langle z, Z^{\perp}\rangle  = 0$.

For the last implication $(v) \Rightarrow (iii)$, note the equivalences
$$ Z^{\perp} \cap C^+(H^{\perp}) = \emptyset \Leftrightarrow (Z^{\perp} \cap H^{\perp}) \cap C^+(H^{\perp}) = \emptyset \Leftrightarrow ((Z+H)^{\perp} \cap H^{\perp}) \cap C^+(H^{\perp}) = \emptyset \; ,$$
where for the last equivalence, we used that $Z^{\perp} \cap H^{\perp} = (Z+H)^{\perp}$.
By Lemma \ref{mink-lemma} (iii), if $ ((Z+H)^{\perp} \cap H^{\perp}) \cap C^+(H^{\perp}) = \emptyset $, then $((Z+H) \cap H^{\perp}) \cap C^+(H^{\perp}) \neq \emptyset $, so
$q^+((H+Z) \cap H^{\perp}) \geq 1$.

\end{proof}

\begin{lemma} \label{Zoext-lemma} With the notations above, the following are equivalent:

(i) $H$ is $Z$-orthogonally extendable;

(ii) $q^+(Z \cap H^{\perp}) \geq 1$;

(iii) $ \pi^{H^{\perp}}(Z^{\perp}) \cap C^+(H^{\perp}) = \emptyset $.
\end{lemma}

\begin{proof}
The equivalence $(i) \Leftrightarrow (ii)$ is obvious. We prove $ (ii) \Rightarrow (iii)$.
Assume there is $z \in Z \cap H^{\perp}$, with $\langle z,z\rangle  \; \; > 0$. Then $\pm z \in C^+(H^{\perp})$
and without loss of generality, assume that $z \in C^+(H^{\perp})$. Since, moreover $z$ is not a null-vector,
by Lemma \ref{mink-lemma} it follows that $\langle z, u \rangle  \; \; > 0$, for all $ u \in C^+(H^{\perp})$.

Let $w \in Z^{\perp}$ and let $w = w^{H} + w^{H^{\perp}}$, where $w^H = \pi^{H}(w)$ and $ w^{H^{\perp}} = \pi^{H^{\perp}}(w)$.
Since $z \in Z \cap H^{\perp}$, $\langle  z, w^{H^{\perp}} \rangle  = \langle z, w\rangle  = 0$. Thus relation $(iii)$ must hold.

The implication $(iii) \Rightarrow (ii)$ follows directly from Lemma \ref{mink-lemma} (iii), simply noting
that the orthogonal complement of $\pi^{H^{\perp}}(Z^{\perp})$ in $H^{\perp}$ is just $Z \cap H^{\perp}$.

\end{proof}

We are now ready to prove Proposition \ref{lath1}.

\vspace{0.2cm}
{\it Proof of Proposition \ref{lath1}:} $(i) \Rightarrow (ii)$
Let $H$ be a $(k-1)$-dimensional positive plane. Assume that $H$ is $Z$-extendable. By Lemma \ref{Zext-lemma} this is equivalent with
$Z^{\perp} \cap C^+(H^{\perp}) = \emptyset$. Assume also that $H$ is not $Z$-orthogonally extendable. By Lemma \ref{Zoext-lemma} this means that there exists
$u \in Z^{\perp}$, so that $\pi^{H^{\perp}}(u) \in C^+(H^{\perp})$. For simplicity, denote $u^{H^{\perp}} = \pi^{H^{\perp}}(u)$ and $u^{H} = \pi^{H}(u)$.
Obviously,
$$ \langle u, u\rangle  = \langle u^H , u^H \rangle  + \langle u^{H^{\perp}}, u^{H^{\perp}}\rangle  \; \; .$$
Since $u \in Z^{\perp}$, the assumption $q^+(Z^{\perp}) = 0$ implies that $\langle u, u\rangle  \; \leq 0$. On the other hand, $\langle u^H , u^H \rangle  \; \geq 0$ since $H$ is positive definite
and $\langle u^{H^{\perp}}, u^{H^{\perp}}\rangle  \; \geq 0$ since $u^{H^{\perp}} \in C^+(H^{\perp})$. Thus, we must have
$$0 = \langle u, u\rangle  = \langle u^H , u^H \rangle  = \langle u^{H^{\perp}}, u^{H^{\perp}}\rangle  \; \; .$$
Since $H$ is positive definite, it follows that $u^H = 0$. Thus $u = u^{H^{\perp}} \in Z^{\perp} \cap C^+(H^{\perp})$. This contradicts $Z^{\perp} \cap C^+(H^{\perp}) = \emptyset$.
Hence $H$ must be $Z$-orthogonally extendable and the first implication is proved.

\vspace{0.2cm}
$(ii) \Rightarrow (i)$ We'll prove the counter-positive -- assume that $q^+(Z^{\perp}) \geq 1$ and construct a $(k-1)$-positive plane which is
$Z$-extendable, but not $Z$-orthogonally-extendable.
Denote by $r = q^+(Z) \geq 1$. Because of the assumption $q^+(Z^{\perp}) \geq 1$, note that $r \leq k - 1$.
Consider now an $r$-positive plane in $Z$ and let $\{\omega_1, \dots, \omega_r \}$ be an orthonormal basis for this plane. Next,
pick $u\in Z^{\perp}$ with $\langle u,u\rangle  = 1 $ (such $u$ exists by the assumption $q^+(Z^{\perp}) \geq 1$).
Extend $\{\omega_1, \dots, \omega_r, u\}$ to an orthonormal basis of a positive $k$-plane in $V$, $\{\omega_1, \dots, \omega_r, u, \eta_{r+2}, \dots, \eta_{k}\}$.
This is possible since $L=Span\{\omega_1, \dots, \omega_r, u\}$ is a positive $(r+1)$-plane in $V$, hence non-degenerate. One can then pick an orthogonal basis for $L^{\perp}$
and extract $\{ \eta_{r+2}, \dots, \eta_{k} \}$ the positive vectors in this basis.
Define the positive $(k-1)$-plane
$$H = Span \{u+\omega_1, \omega_2, \dots, \omega_r, \eta_{r+2}, \dots, \eta_{k}\} . $$ It is then evident that $H$ is $Z$-extendable because
$$H + \mathbb{R}\omega_1 = Span\{\omega_1, \dots, \omega_r, u, \eta_{r+2}, \dots, \eta_{k}\}. $$
 We claim that $H$ is not $Z$-orthogonally-extendable. Indeed, if $H$ were $Z$-orthogonally-extendable,
then there is a vector $\omega \in Z \cap H^{\perp}$ such that $H+\mathbb{R}\omega$ is a positive $k$-dimensional. By the definition of $H$ and since $\omega \in Z$, $u\in Z^{\perp}$, we have
$$0 = \langle \omega, \omega_i\rangle  \text{ for all $i \geq 2$}, $$
$$0 = \langle \omega, u + \omega_1\rangle  = \langle \omega, u\rangle  + \langle \omega, \omega_1\rangle  = \langle \omega, \omega_1\rangle  \; .$$
Therefore, the span of $\{\omega, \omega_1, \dots, \omega_r\}$ is an $(r+1)$-dimensional positive plane contained in $Z$, contradicting the fact that $q^+(Z) = r$.
This completes the proof of the second implication and, thus, of the theorem. $\Box$

\vspace{0.2cm}

Next we prove Proposition \ref{lath2}.

\vspace{0.2cm}

{\it Proof of Proposition \ref{lath2}:} $(i) \Rightarrow (ii)$ Follows directly from Lemma \ref{Zext-lemma}, as for any $H$ we have $q^+(H+Z) = q^+(Z) = k$.

$(ii) \Rightarrow (i)$ The argument is very similar to the corresponding implication in Proposition \ref{lath1}. Suppose $q^+(Z) = r \leq k-1$. Consider a positive
$(k-1)$-plane $H$ that contains a positive $r$-dimensional subspace of $Z$. Then $H$ cannot be $Z$-extendable,
as this would imply that $Z$ contains a positive $(r+1)$-dimensional plane.

$(i) \Rightarrow (iii)$ Observe that $q^+(Z) = k$ immediately implies $q^+(Z^{\perp}) = 0$ and now use $(i) \Rightarrow (ii)$ and Proposition \ref{lath1}.

As implication $(iii) \Rightarrow (ii)$ is obvious, the proof is complete. $\Box$

\section{Proofs of the main results}

{\it Proof of Theorem \ref{main0}:} Theorem \ref{main0} follows directly from Proposition \ref{lath1}. Just take
$(V, \langle \, \cdot\;, \cdot \, \rangle)$ be $(\Lambda^2(\mathfrak{g}^{\star}), \Phi^{\mu})$ and take the subspace $Z$ be the space of
closed 2-forms $\mathcal{Z}^2$.
Modulo sign (which is not important), an almost complex structure $J$ is identified with its positive 2-plane of $J$-anti-invariant forms $H = \Lambda_J^-$,
and is clear from Proposition \ref{phimu4d} that the $Z$-extendable condition in this context is equivalent to the tame condition for $J$,
while the $Z$-orthogonally-extendable condition is just the compatibility of $J$ with a symplectic form.
Note also that $(\mathcal{Z}^2)^{\perp}$ is isomorphic with the space of boundary vectors $B_2$, via the Riesz map induced by $\Phi^{\mu}$ . $\Box$

\vspace{0.3cm}

{\it Proof of Corollary \ref{cor1}:} Using the classification of symplectic 4-dimensional Lie algebras of Ovando \cite{Ov}, one
checks that the only cases when the condition of Theorem \ref{main0} is not satisfied are
the Lie algebra $\mathfrak{r}_2 \mathfrak{r}_2$ endowed with either orientation and
the Lie algebra $\mathfrak{d }_{4,2}$ endowed with the non-complex orientation.

The first case,  $\mathfrak{g} = \mathfrak{r}_2 \mathfrak{r}_2$ is characterized by a  basis
$\{ e_1, e_2, e_3, e_4 \}$ such that $[e_1, e_2] = e_2, \; [e_3, e_4] = e_4$. (It is the Lie algebra
of the group $Aff(\mathbb{R}) \times Aff(\mathbb{R})$, where $Aff(\mathbb{R})$ is the Lie group of affine transformations of the Euclidean line
and $\mathfrak{r}_2 = \mathfrak{aff}(\mathbb{R})$ is the unique non-abelian 2-dimensional Lie algebra.)
Alternatively, if $\{ e^1, e^2, e^3, e^4 \}$ is the dual basis
of $\mathfrak{g}^{\star} = \Lambda^1(\mathfrak{g}^{\star})$,
$$ de^1 = 0 \; \; \; de^2 = - e^{12} \; \; \; de^3 = 0 \; \; \; de^4 = - e^{34} \; .$$
It is easily checked that the spaces of closed 2-forms and boundary 2-vectors are given by
$$
\mathcal{Z}^2 = Span( e^{12} , e^{13}, e^{34}) \; \; B_2 = Span( e_{14} , e_{23}, e_{24}), \; \; .$$
Note that if the orientation on $\mathfrak{g}$ is given by
$\mu = e^{1234}$, then $\Phi_{\mu}(e_{14} - e_{23}, e_{14} - e_{23}) = -2 < 0$,
and if $\mu = - e^{1234}$, then $\Phi_{\mu}(e_{14} + e_{23}, e_{14} + e_{23}) = -2 < 0$.

Here is a concrete example of an almost complex structure on $\mathfrak{r}_2 \mathfrak{r}_2$ (with orientation $\mu = e^{1234}$),
which is tamed but not compatible.  Let $J$ be given by
$$\Lambda^-_J = Span \Big((e^{12} +e^{34}) + (e^{14} +e^{23}),
 (e^{13} - e^{24}) \Big) .$$ On vectors, $J$ is explicitly given by
$$ Je_1 = \frac{1}{\sqrt{2}}(e_2 - e_4),  \; \; Je_2  = \frac{1}{\sqrt{2}}(-e_1 - e_3),
\; \; Je_3  = \frac{1}{\sqrt{2}}(e_4 + e_2), \; \; Je_4  = \frac{1}{\sqrt{2}}(-e_3 + e_1) .$$
Then $J$ is tamed by the symplectic form $\omega_0 = e^{12}+e^{34}$. On the other hand, $J$ is not compatible
with any symplectic form. Indeed, a positively oriented symplectic form on $\mathfrak{r}_2 \mathfrak{r}_2$ with orientation $\mu = e^{1234}$ is
of the form
$$ \omega = a e^1 \wedge e^2 + b e^1 \wedge e^3 + c e^3 \wedge e^4, \; \mbox{ with } a,b,c \in \mathbb{R}, ac > 0 . $$
But the condition that $\omega$ be orthogonal to $\Lambda^-_J$ is $ a+c = 0, \; b= 0$, thus, $ac > 0$ cannot be satisfied.

\vspace{0.1cm}

The second case, the Lie algebra $ \mathfrak{d }_{4,2}$ is given by the non-zero Lie brackets:
$$ \mathfrak{d }_{4,2} : \; \; [e_1, e_2] = e_3, \; \; [e_4, e_3] = e_3, \; \; [e_4, e_1] = 2e_1, \; \; [e_4, e_2] = - e_2 . \hspace{6cm}$$
The spaces of closed two-forms $\mathcal{Z}^2$ and boundary two-vectors $B_2$ are given respectively by
$$ \mathcal{Z}^2 = Span(e^{12} - e^{34}, e^{14}, e^{23}, e^{24}), \; \; B_2 = Span (e_{12}+ e_{34}, e_{13}) .$$
Note that if the orientation is given by $\mu = - e^{1234}$ then $\Phi_{\mu}$ is non-positive definite on $B_2$, hence with this orientation
$\mathfrak{d }_{4,2}$ satisfies the tame-compatible property.
However, if the orientation is given by $\mu =  e^{1234}$ , then $\Phi_{\mu}(e_{12}+e_{34}, e_{12}+e_{34}) > 0$, so according to Theorem \ref{main0}
there are almost complex structures inducing this orientation which are tamed but not compatible with symplectic forms.
A concrete example of such $J$ can be again given by
$$\Lambda^-_J = Span \Big((e^{12} +e^{34}) + (e^{14} +e^{23}), (e^{13} - e^{24}) \Big) .$$
As in the case above, the reader can check directly that $J$ is tamed by, but not compatible with symplectic forms on $ \mathfrak{d }_{4,2}$, with the orientation given by $\mu =  e^{1234}$. $\Box$

\vspace{0.2cm}

{\bf Remark:} Using the classification of Ovando \cite{Ov}, in \cite{EF} it was shown that Question \ref{Donaldson} has an affirmative answer for
4-dimensional Lie algebras endowed with an {\it integrable} almost complex structure $(\mathfrak{g}, J)$.
Certainly Theorem \ref{main0} offers an alternative way of eliminating most cases, but note that $\mathfrak{r}_2 \mathfrak{r}_2$
does admit complex structures. One checks separately that these complex structures are compatible with symplectic forms,
so $\mathfrak{r}_2 \mathfrak{r}_2$ carries K\"ahler structures, but also carries almost complex structures which are tamed
 and non-compatible.

\vspace{0.2cm}

As we mentioned in the introduction, a natural further question is if Theorem \ref{main0} and Corollary \ref{cor1} have any kind of extension
in higher dimensions. We do not know the answer and leave this for further work. The linear algebra set up of section 3 applies only to the 4-dimensional tame-compatible problem, as we use the description of almost complex structures via positive planes of 2-forms, which is particular to dimension 4.

\vspace{0.2cm}

We end with the observation that abelian Lie algebras are the only oriented even-dimensional Lie algebras with the
property that all almost complex structures are tamed by (and compatible with) symplectic forms.
\begin{prop} \label{main1}
If $\mathfrak{g}$ is a $2n$-dimensional ($2n \geq 4$), non-abelian, oriented Lie algebra, then
there exists an almost complex structure $J$ on $\mathfrak{g}$ which is not tamed by a symplectic form.
\end{prop}
\begin{proof} If the Lie algebra $\mathfrak{g}$ does not carry symplectic forms then the statement is obvious.
If the Lie algebra $\mathfrak{g}$ carries symplectic forms then $\mathfrak{g}$ must be solvable \cite{LM},
hence, the center $\xi$ is nontrivial. We then use Lemma 3.2 of \cite{EFV},
which we restate here for the convenience of the reader.  Note that in the original statement in \cite{EFV}
$J$ is assumed complex, but integrability is not used anywhere in the proof:

\vspace{0.2cm}
{\bf Lemma:} (Lemma 3.2 in \cite{EFV}) {\em Let $\mathfrak{g}$ be a real Lie algebra endowed with an almost complex structure $J$ such that
$J\xi \cap [\mathfrak{g}, \mathfrak{g}] \neq \{0\}$, where $\xi$ denotes the center of $\mathfrak{g}$.
Then $(\mathfrak{g}, J)$ cannot admit a tamed symplectic structure.}

\vspace{0.2cm}
Now Proposition \ref{main1} is immediate. Pick non-zero vectors $u \in \xi$, $v\in  [\mathfrak{g}, \mathfrak{g}]$, define
$Ju = v, \; Jv = - u$ and then extend $J$ to an almost complex structure on $\mathfrak{g}$. By the Lemma, $J$ cannot be tamed.
\end{proof}

{\bf Remark:}  In dimension 4, Proposition \ref{main1} can be also seen as a consequence of Proposition \ref{lath2}, with the additional observation
that a 4-dimensional Lie algebra $\mathfrak{g}$ with $q^+( \mathcal{Z}^2) = 3$ must be abelian. The observation is true because the condition $q^+( \mathcal{Z}^2) = 3$
implies the existence of a triple of symplectic forms $\omega_i$ orthogonal with respect to $\Phi^{\mu}$ and with $\omega_i^2 = \mu$. Then Hitchin's lemma \cite{Hi}
implies that $\mathfrak{g}$ must carry a hyperK\"ahler structure. But in dimension four, the only such Lie algebra is the abelian one, as it follows from the general description of hyperK\"ahler Lie algebras given by \cite{BDF}.

\end{document}